\documentclass[a4paper,11pt]{amsart}
\usepackage[utf8x]{inputenc}
\usepackage{amsthm,amsfonts,amsmath,amssymb,enumerate}
\usepackage{verbatim} 
\usepackage{longtable}

\usepackage{hyperref}
\usepackage{pdfsync}
\renewcommand{\L}{\mathcal{L}}

\newcommand{\C}{\mathbb{C}}
\newcommand{\g}{\mathfrak{g}}
\newcommand{\n}{\mathfrak{n}}
\newcommand{\h}{\mathfrak{h}}
\renewcommand{\sl}{\mathfrak{sl}}

\renewcommand{\deg}{\operatorname{deg}}
\renewcommand{\O}{\mathcal{O}}

\newcommand {\Der} {\operatorname{Der}}

\newtheorem{thm}{Theorem}
\newtheorem{theorem}{Theorem}[section]

\newtheorem{proposition}[theorem]{Proposition}

\theoremstyle{definition}

\newtheorem{example}[theorem]{Example}

\theoremstyle{remark}

\numberwithin{equation}{section}

\begin{document}

\title[The Grunewald-O'Halloran conjecture]{The Grunewald-O'Halloran conjecture \\ for nilpotent Lie algebras of rank $\ge 1$}
\author{Joan Felipe Herrera-Granada and  Paulo Tirao}
\address{CIEM-FaMAF, Universidad Nacional de C\'ordoba, Argentina}
\date{October 31, 2013}
\subjclass[2010]{Primary 17B30; Secondary 17B99}
\keywords{Nilpotent Lie algebras, Vergne's conjecture, Grunewald-O'Halloran conjecture, degenerations, deformations.}

\maketitle

\begin{abstract}
Grunewald and O'Halloran conjectured in 1993 that every complex nilpotent Lie algebra is the degeneration of another, 
non isomorphic, Lie algebra.
We prove the conjecture for the class of nilpotent Lie algebras admitting a semisimple derivation,
remaining open for the class of characteristically nilpotent Lie algebras.
In dimension 7, where the first characteristically nilpotent Lie algebras appear, we prove the conjecture 
and we also exhibit explicit nontrivial degenerations to every 7-dimensional nilpotent Lie algebra.
\end{abstract}

\section{Introduction}

The study of the algebraic varieties of Lie algebras, solvable, and nilpotent Lie algebras of dimension $n$ 
turned out to be a very hard subject.
The theory of deformations of algebras started with a series of papers by Gerstenhaber,
the first being \cite{G}.
Since then a lot of efforts has been done (see for instance \cite{NR1,R,NR2,V,C1,K}), 
however many natural questions remain unsolved.
For example, the determination of the irreducible components of the variety of nilpotent Lie algebras
seems today out of reach.

Among the open questions there are two conjectures about nilpotent Lie algebras.
One, due to Grunewald and O'Halloran \cite{GO2}, states that every 
complex nilpotent Lie algebra is the degeneration of another, non isomorphic, Lie algebra.
The other one, known as Vergne's conjecture, states that there are no rigid complex nilpotent Lie algebras 
in the algebraic variety $\L_n$ of complex Lie algebras of dimension $n$.
Meaning that there are no nilpotent Lie algebras with open orbit in $\L_n$, that is such that their isomorphisms classes are open in $\L_n$.
The first conjecture is a priori stronger than the second one.
In this short paper we address the Grunewald-O'Halloran conjecture.

It is well known that, over fields of characteristic zero, geometric rigidity is equivalent to formal rigidity,
the latest meaning that all formal deformations are trivial \cite{GS}.
However, this does not imply that the Grunewald-O'Halloran conjecture and Vergne's conjecture are equivalent.
If so, it would also imply that every non geometrically rigid Lie algebra is the degeneration of another non isomorphic Lie algebra,
which is not true already in dimension $n=3$.
In fact the only complex rigid Lie algebra of dimension 3 is the simple Lie algebra $\sl_2(\C)$ and, for instance, the solvable (non nilpotent) Lie algebra
$\mathfrak{r}+\C$, where $\mathfrak{r}$ is the 2-dimensional solvable Lie algebra, is on top of the Hasse diagram of degenerations,
and in particular it is not the degeneration of any other Lie algebra (see \cite{CD} and \cite{BSt}).

Complex Lie algebras and nilpotent Lie algebras of small dimension are classified and
in this cases all the degenerations among them and also which are rigid is known.
All degenerations that occur among complex Lie algebras of dimension $\le 4$ are given in \cite{St} and \cite{BSt}.
In \cite{GO1} and \cite{Se} all degenerations for complex nilpotent Lie algebras of dimension 5 and 6 are given
and more recently, in \cite{B}, some degenerations for some 5-step and 6-step complex nilpotent Lie algebras of dimension 7 are given.
Results on the different varieties and on rigidity in low dimensions may be found in \cite{CD,C2}.
In \cite{AG} and \cite{AGGV} the components of the varieties of nilpotent Lie algebras of dimension 7 and 8 are given.

Carles \cite{C1} investigated the structure of rigid Lie algebras over algebraically closed fields of characteristic zero.
In particular he proved that nilpotent Lie algebras of rank $\ge 1$ are never rigid and moreover
nilpotent Lie algebras with a codimension 1 ideal of rank $\ge 1$ are also never rigid.
That is, Vergne's conjecture holds for this class, remaining open for characteristically nilpotent Lie algebras
for which all its ideals of codimension 1 are also characteristically nilpotent.

In the paper \cite{GO2}, the authors constructed nontrivial linear deformations for large classes of nilpotent Lie algebras and left open 
the question of which of those deformations correspond to degenerations.
Their construction of linear deformations of a given Lie algebra $\g$,
relies on the existence of a codimension 1 ideal $\h$ of $\g$ with a semisimple derivation $D\in \Der(\h)$,
and applies not only to nilpotent Lie algebras.
In general, the deformations constructed do not correspond to a degeneration.
A fixed ideal $\h$ may produce many non equivalent deformations, some of which may correspond to a degeneration
and some may not. 

We prove two things.
On the one hand we prove that the Grunewald-O'Halloran conjecture holds for nilpotent Lie algebras of rank $\ge 1$,
leaving it open for characteristically nilpotent Lie algebras.
On the other hand, we prove that the conjecture holds for 7-dimensional nilpotent Lie algebras and moreover
and interesting for us we exhibit explicit degenerations to each 7-dimensional nilpotent Lie algebra.

More precisely, we show that if the semisimple derivation $D$ of $\h$ is the restriction to $\h$ of a semisimple derivation of $\g$,
then the associated deformation does correspond to a degeneration.
Then we are able to prove the following.

\begin{thm}\label{thm:1}
If $\n$ is a complex nilpotent Lie algebra with a nontrivial semisimple derivation, then $\n$ is the degeneration of another,
non isomorphic, Lie algebra. 
\end{thm}

The first characteristically nilpotent Lie algebras appear in dimension 7.
Hence, by Theorem \ref{thm:1}, the Grunewald-O'Halloran conjecture holds in dimension $<7$.
Complex nilpotent Lie algebras of dimension 7 are classified:
there are infinitely many isomorphism classes and infinitely many of them are characteristically nilpotent.
We shall refer to the classification by Magnin \cite{M}.
We work out this family on a case by case basis, by considering linear deformations constructed after choosing suitable
codimension 1 ideals and particular derivations of them, proving the following result.

\begin{thm}\label{thm:2}
Every complex nilpotent Lie algebra of dimension $\le 7$, is the degeneration of another, non isomorphic, Lie algebra.
\end{thm}

We note that the variety of complex nilpotent Lie algebras of dimension 7 has two components, each of which is the closure of the orbit
of a family of Lie algebras \cite{AG}.
Being degeneration transitive, to proof Theorem \ref{thm:2} it is enough to find nontrivial degenerations to these two families.
In the case of dimensions $<7$ this argument reduces the proof to finding a nontrivial degeneration to a single algebra.
This is easy to do and for completeness we do it in dimension 6.

In this paper all Lie algebras will be over the complex numbers.

\section{Linear deformations and degenerations}

Let $\L_n$ be the algebraic variety of complex Lie algebras of dimension $n$,
that is the algebraic variety of Lie brackets $\mu$ on $\C^n$ ($\L_n \subseteq \C^{n^3}$).
Given a complex Lie algebra $\g=(\C^n,\mu)$, we shall refer to it indistinctly by $\g$, $(\g,\mu)$ or $\mu$.
The group $GL_n=GL_n(\C)$ acts on $\L_n$ by `change of basis':
\[ g\cdot \mu (x,y)=g(\mu(g^{-1}x,g^{-1}y)), \qquad g\in GL_n. \]
Thus the orbit $\O(\mu)$ of $\mu$ in $\L_n$, is the isomorphism class of $\mu$. 

A Lie algebra $\mu$ is said to degenerate to a Lie algebra $\lambda$, denoted by $\mu \rightarrow_{\deg} \lambda$, if $\lambda\in\overline{\O(\mu)}$,
the Zariski closure of $\O(\mu)$. If $\lambda\not\simeq\mu$, then $\lambda$ is in the boundary of the orbit $\O(\mu)$ but outside it.
Since the Zariski closure of $\O(\mu)$ coincides with its closure in the relative topology of $\C^{n^3}$,
if $g:\C^\times\rightarrow GL_n$, $t\mapsto g_t$, is continuous  
and $\lim_{t\mapsto 0}g_t\cdot \mu=\lambda$, then $\mu\rightarrow_{\deg}\lambda$.
The degeneration $\mu\rightarrow_{\deg}\lambda$ is said to be realized by a 1-PSG, 
if $g_t$ is a 1-parameter subgroup as a morphism of algebraic groups. 
Recall that if $g_t$ is a 1-PSG, then $g_t$ is diagonalizable with eigenvalues $t^{m_i}$ for some integers $m_i$.

A linear deformation of a Lie algebra $\mu$ is, for the aim of this paper, a family $\mu_t$, $t\in \C^\times$,
of Lie algebras such that
\[ \mu_t=\mu + t\phi, \]
where $\phi$ is a skew-symmetric bilinear form on $\C^n$.
It turns out that $\mu_t$ is a linear deformation of $\mu$ if and only if $\phi$ is a Lie algebra bracket which in addition 
is a 2-cocycle of $\mu$.

If a given a linear deformation $\mu_t$ of $\mu$ is such that $\mu_t\in\O(\mu_1)$ for all $t\in\C^\times$, 
then $\mu_1\rightarrow_{\deg}\mu$.
In fact, for each $t\in\C^\times$ there exist $g_t\in GL_n$ such that $g_t^{-1}\cdot \mu_1=\mu_t$,
then $\lim_{t\mapsto 0}g_t^{-1}\cdot \mu_1=\lim_{t\mapsto 0}\mu_t=\mu$.
Hence, in order to show that $\mu_1\rightarrow_{\deg}\mu$, one only needs to prove that 
for each $t\in\C^\times$ there exist $g_t\in GL_n$ such that
\begin{equation}\label{eqn:degeneration}
 \mu_1(g_t(x),g_t(y)))=g_t(\mu_t(x,y)), \quad\text{for all $x,y\in\C^n$}.
\end{equation}

\subsection{Construction of linear deformations}

We recall now the construction of linear deformations in \cite{GO2}.

Let $(\g,\mu)$ be a given Lie algebra of dimension $n$ and let $\h$ be a codimension 1 ideal of $\g$ with a semisimple derivation $D$.
For any element $X$ of $\g$ outside $\h$, $\g=\langle X \rangle \oplus \h$.
The bilinear form $\mu_D$ on $\g$ defined by $\mu_D(X,z)= D(z)$ and $\mu_D(y,z)=0$, for $y,z\in\h$, is a 2-cocycle for $\mu$
and a Lie bracket.
Hence,
\begin{equation}\label{eqn:linear-def}
 \mu_t=\mu + t\mu_D,
\end{equation}
is a linear deformation of $\mu$.
If $\g$ is nilpotent, then $\mu_t$ is always solvable but not nilpotent. 
In particular, $\mu_t$ is not isomorphic to $\mu$ for all $t\in\C^\times$.
The construction described above can be carried out also for any derivation $D$, not necessarily semisimple.
However, one can not assure that $\mu_t$ is not isomorphic to $\mu$ in this case.

\subsection{Degenerations from deformations}

Under certain hypothesis on the derivation $D$, the deformation constructed above does correspond to a degeneration.

\begin{proposition}\label{prop:restriction}
 Let $\n$ be a nilpotent Lie algebra with an ideal $\h$ of codimension 1 admitting a nontrivial semisimple derivation $D$.
 If $D$ is the restriction of a semisimple derivation $\tilde D$ of $\n$ 
 such that it is nontrivial on a direct invariant complement of $\h$,
 then $\n$ is the degeneration of another, non isomorphic, Lie algebra.
 Moreover, the degeneration can be realized by a 1-PSG.
\end{proposition}

\begin{proof}
Let $\n=(\n,\mu)$.
Let $X$ be an eigenvector of $\tilde D$ complementary to $\h$ and let $\lambda_0\ne 0$ be its eigenvalue.
We may assume that $\lambda_0=1$ (by considering $\tilde D/\lambda_0$ and $D/\lambda_0$ instead of $\tilde D$ and $D$).

Let $\lambda_1,\dots,\lambda_k$ be the different eigenvalues of $D$ and let $\h=\h_{\lambda_1}\oplus\dots\oplus \h_{\lambda_k}$
be the corresponding graded decomposition of $\h$, 
that is $\mu(\h_{\lambda_i},\h_{\lambda_j})\subseteq \h_{\lambda_i+\lambda_j}$.

Hence,
\[ \n=(\langle X\rangle \oplus \h,\mu) \]
where both summands of $\n$ are $\tilde D$-invariant and $\mu(X,\h_{\lambda_j})\subseteq \h_{1+\lambda_j}$.

Let $\mu_t=\mu+t\mu_D$ be the linear deformation constructed as in \eqref{eqn:linear-def}, which is given by
\begin{align*}
\mu_{t}(X,y_{j}) &= \mu(X,y_j)+t \lambda_j y_j,\quad \text{if $y_j\in \h_{\lambda_j}$, for $1\le j \le k$}. \\
\mu_{t}(y_i,y_j) &= \mu(y_i,y_j),\quad \text{if $y_i\in\h_{\lambda_i}$ and $y_j\in\h_{\lambda_j}$, for $1\le i,j\le k$}.
\end{align*}
 
Let $g_t\in GL_n$, where $n=\dim \n$, be defined by
\[ g_t|_{\langle X \rangle}=t I  \qquad\text{and}\qquad g_t|_{\h_{\lambda_i}}=t^{\lambda_i} I, \quad\text{for $i=1\dots k$}.  \] 
It is not difficult to check that \eqref{eqn:degeneration} is satisfied.
In fact, if $y_i\in\h_{\lambda_i}$ and $y_j\in\h_{\lambda_j}$ for $1\le i,j\le k$, then  
\begin{eqnarray*}
 g_t(\mu_{t}(X,y_j))       &=& g_t(\mu(X,y_j)+\lambda_j ty_j)=t^{1+\lambda_j}\mu (X,y_j)+\lambda_jt^{\lambda_j+1}y_j, \\
 \mu_{1}(g_t(X),g_t(y_j)) &=& \mu_{1}(t X,t^{\lambda_j}y_j)=t^{1+\lambda_j}\mu(X,y_j)+\lambda_jt^{\lambda_j+1}y_j,\\ 
\end{eqnarray*}
and 
\begin{eqnarray*}
 g_t(\mu_{t}(y_i,y_j)) &=& g_t(\mu(y_i,y_j))=t^{\lambda_i+\lambda_j}\mu (y_i,y_j), \\
 \mu_{1}(g_t(y_i),g_t(y_j)) &=& \mu_{1}(t^{\lambda_i}y_i,t^{\lambda_j}y_j)=t^{\lambda_i+\lambda_j}\mu(y_i,y_j). \\ 
\end{eqnarray*}
Therefore, being $\mu_1$ solvable,  $\mu$ is the degeneration of another, non isomorphic, Lie algebra.
\end{proof}

In the above proposition the ideal $\h$ is given, but clearly any such ideal will work.
Hence, if $\tilde D$ is a derivation of $\n$ that preserves an ideal $\h$ and such that
its restriction to $\h$ is semisimple, we get for $\n$ the same conclusion of Proposition \ref{prop:restriction}.
This is the statement in Theorem \ref{thm:1}.

\begin{proof}[Proof of Theorem \ref{thm:1}]
 The semisimple derivation $D$ of $\n$ preserves the (characteristic) ideal $[\n,\n]$.
 Let $V$ be a $D$-invariant complement of $[\n,\n]$ and let $\{X_1,\dots,X_r\}$ be a basis of $V$ formed by eigenvectors of $D$.
 Since $V$ generates $\n$ as a Lie algebra (see for instance \cite{J}, page 29) and $D$ is nontrivial, $D$ is nontrivial on $V$ 
 and we may assume that $X_1$ is an eigenvector with nonzero eigenvalue. 
 Now let $\h=\langle X_2,\dots,X_r \rangle \oplus [\n,\n]$.
 Clearly $\h$ is an ideal of $\n$ of codimension 1, $D$ preserves $\h$, $D|\h$ is semisimple 
 and $D$ is nontrivial on $X_1$.
 Therefore, by Proposition \ref{prop:restriction}, $\n$ is the degeneration of a Lie algebra non isomorphic to $\n$.
\end{proof}

\section{The conjecture in dimension 7}

All nilpotent Lie algebras of dimension $<7$ have semisimple derivations.
Therefore the Grunewald-O'Halloran conjecture holds in this case.

Moreover, in dimensions 2, 3, 4, 5 and 6 all nilpotent Lie algebras (finite number of isomorphism classes)
are the degeneration of a single one \cite{GO1,Se}.
Hence, an algebra degenerating to it degenerates to all the others as well.

By considering different linear deformations, we found that each nilpotent Lie algebra of dimension $<7$
is the degeneration of many others, non isomorphic, Lie algebras.
Many of those degenerations can be realized by a 1-PSG, but others can not.

\begin{example}
The 6-dimensional nilpotent Lie algebra $12346_E$ in \cite{Se}, that we rename $\mu$, defined by
\begin{align}\label{ec}
 \mu(e_1,e_2)&=e_3, & \quad \mu(e_1,e_3)&=e_4, & \quad \mu(e_1,e_4)&=e_5, \\
 \mu(e_2,e_3)&=e_5, & \quad \mu(e_2,e_5)&=e_6, & \quad \mu(e_3,e_4)&=-e_6,\nonumber
\end{align}
degenerates to all other nilpotent Lie algebras of dimension 6 \cite{Se}.

We now construct a solvable linear deformation of $\mu$ that degenerates to it, and therefore
to all other 6-dimensional nilpotent Lie algebras.
To this end consider the ideal $\h=\langle e_2,e_3,e_4,e_5,e_6 \rangle$ and the derivation $D$ of $\h$ defined by
\[ D(e_2)=e_2, \quad D(e_4)=2e_4, \quad D(e_5)=e_5, \quad D(e_6)=2e_6. \]
This produces the 2-cocycle $\mu_{D}$, defined by
\[
 \mu_{D}(e_1,e_2)=e_2, \quad \mu_{D}(e_1,e_4)=2e_5, \quad \mu_{D}(e_1,e_5)=e_5, \quad \mu_{D}(e_1,e_6)=2e_6.
\]
The corresponding deformation of $\mu$, $\mu_t=\mu+t\mu_{D}$, is then given by
\begin{align*}
 \mu_{t}(e_1,e_2)&=e_3+te_2, \quad & \mu_{t}(e_1,e_3)&=e_4, \quad & \mu_{t}(e_1,e_4)&= e_5+2te_4,\\
 \mu_{t}(e_1,e_5)&=te_5, \quad & \mu_{t}(e_1,e_6)&=2te_6,  \quad & \mu_{t}(e_2,e_3)&=e_5,\\
 \mu_{t}(e_2,e_5)&=e_6, \quad & \mu_{t}(e_3,e_4)&=-e_6,
\end{align*}
and in particular $\mu_{1}$ is given by
 \begin{align*}
 \mu_{1}(e_1,e_2)&=e_3+e_2, \quad & \mu_{1}(e_1,e_3)&=e_4, \quad & \mu_{1}(e_1,e_4)&=e_5+2e_4,\\ 
 \mu_{1}(e_1,e_5)&=e_5, \quad & \mu_{1}(e_1,e_6)&=2e_6, \quad & \mu_{1}(e_2,e_3)&=e_5,\\
 \mu_{1}(e_2,e_5)&=e_6, \quad & \mu_{1}(e_3,e_4)&=-e_6.
\end{align*} 
Let $g_t\in GL_6$ be the 1-PSG given by
\[ 
 g_t=\begin{pmatrix} t \\ & t^2 \\ & & t^3 \\ & & & t^4 \\ & & & & t^5 \\ & & & & & t^7 \end{pmatrix}. 
\]
It is easy to verify that, for all $t\ne 0$, $g_t^{-1}\cdot \mu_{1} =\mu_{t}$ and thus $\mu_1\rightarrow_{\deg} \mu$.  

\end{example}


The variety of nilpotent Lie algebras of dimension 7 has two irreducible components,
each of which is the closure of the orbits of two families $\mu_\alpha^1$ and $\mu_\alpha^2$, with $\alpha\in\C$ \cite[Main Theorem]{AG}.
The first family is made of nilpotent Lie algebras of rank $\ge 1$, while the second family is made entirely of
characteristically nilpotent algebras.

By Theorem \ref{thm:1} and being degeneration transitive, to prove Theorem \ref{thm:2} it suffices to find
for each algebra in the second family another non isomorphic Lie algebra degenerating to it.
All algebras $\mu_\alpha^2$ are indecomposable, because in dimension 7 all the decomposables are of rank $\ge 1$.
In what follows we refer to the classification by Magnin \cite{M}.
Here the (indecomposable) characteristically nilpotent Lie algebras
are given as a continuous family and seven isolated algebras: 
\[ \g_{7,0.1}\quad \g_{7,0.2}\quad \g_{7,0.3}\quad \g_{7,0.4(\lambda)}\quad \g_{7,0.5}\quad \g_{7,0.6}\quad \g_{7,0.7}\quad \g_{7,0.8} \]
Without identifying the algebras $\mu_\alpha^2$ within this classification, Theorem \ref{thm:2} follows if we are able to construct for each of 
these another Lie algebra degenerating to it.
Notice that by doing this, we are exhibiting for each 7-dimensional characteristically nilpotent Lie algebra, another one degenerating non-trivially to it,
something that we find interesting in itself.

\begin{proof}[Proof of Theorem \ref{thm:2}]
We start by considering the family $\g_{7,0.4(\lambda)}$, which is defined by
\begin{align*}
 \mu(e_{1},e_{2})&=e_{3}, \quad \mu(e_{1},e_{3})=e_{4}, \quad \mu(e_{1},e_{4})=e_{6}+\lambda e_{7}, \\
 \mu(e_{1},e_{5})&=e_{7}, \quad \mu(e_{1},e_{6})=e_{7}, \quad \mu(e_{2},e_{3})=e_{5}, \\
 \mu(e_{2},e_{4})&=e_{7}, \quad \mu(e_{2},e_{5})=e_{6}, \quad \mu(e_{3},e_{5})=e_{7}.
\end{align*}
Take the ideal $\h=\langle e_{2},e_{3},e_{4},e_{5},e_{6},e_{7}\rangle$ and $D\in\Der(\h)$ defined by
\[ D(e_{2})=e_{2}\quad D(e_{5})=e_{5}\quad D(e_{6})=2e_{6},\quad D(e_{7})=e_{7}. \]
The corresponding 2-cocycle $\mu_{D}$ is given by 
\[ \mu_{D}(e_{1},e_{2})=e_{2},\quad \mu_{D}(e_{1},e_{5})=e_{5},\quad \mu_{D}(e_{1},e_{6})=2e_{6},\quad \mu_{D}(e_{1},e_{7})=e_{7}, \]
and the corresponding deformation $\mu_{t}=\mu + t\mu_{D}$ of $\mu$ is given by
\begin{align*}
 \mu_{t}(e_{1},e_{2})&=e_{3}+te_{2}, \quad & \mu_{t}(e_{1},e_{3})&=e_{4}, \quad & \mu_{t}(e_{1},e_{4})&=e_{6}+\lambda e_{7}, \\
 \mu_{t}(e_{1},e_{5})&=e_{7}+te_{5}, \quad & \mu_{t}(e_{1},e_{6})&=e_{7}+2te_{6}, & \quad \mu_{t}(e_{1},e_{7})&=te_{7}, \\
 \mu_{t}(e_{2},e_{3})&=e_{5}, \quad & \mu_{t}(e_{2},e_{4})&=e_{7}, \quad & \mu_{t}(e_{2},e_{5})&=e_{6},\\
 \mu_{t}(e_{3},e_{5})&=e_{7}.
\end{align*}
Consider now $g_t=g_{t}(\lambda)\in GL_7$ given by
\[g_t=\left(\begin{smallmatrix}
     t & 0 & 0 & 0 & 0 & 0 & 0\\[0.1cm] 
     0 & 1 & 0 & 0 & 0 & 0 & 0\\[0.1cm]
     0 & 0 & t & 0 & 0 & 0 & 0\\[0.1cm]
     0 & 0 & 0 & t^{2} & 0 & 0 & 0\\[0.1cm]
     \frac{1}{4}\left(\frac{t^{2}-1}{t}\right) & \left(1-\lambda+\frac{\lambda}{t}-\frac{1}{t^{2}}\right) & 0 & 0 & t & 0 & 0\\[0.1cm]
     0 & 0 & \frac{1}{4}\left(\frac{1-t^{2}}{t}\right) & \frac{1}{2}\left(1-t^{2}\right) & 0 & t & 0\\[0.1cm]
     0 & 0 & \left(t-\lambda t+\lambda-\frac{1}{t}\right) & \left(\frac{1}{2}t^{2}-\lambda t^{2}+\lambda t-\frac{1}{2}\right)& \left(\lambda t-t-\lambda + \frac{1}{t}\right) & 0 & t^{2}
     \end{smallmatrix}\right).
\]
The calculations below show that $g_t^{-1}\cdot\mu_1=\mu_t$ and thus $\mu_{1}\rightarrow_{deg}\mu$.

\begin{footnotesize}
\begin{flalign*}
\bullet \ g_{t}\mu_{t}(e_{1},e_{2})&=g_{t}(e_{3}+te_{2})&\\
                                   &=te_{3}+\frac{1}{4}\left(\frac{1-t^{2}}{t}\right)e_{6}+\left(t-\lambda t+\lambda-\frac{1}                   {t}\right)e_{7}+te_{2}+t\left(1-\lambda+\frac{\lambda}{t}-\frac{1}{t^{2}}\right)e_{5}&\\
                                   &=te_{2}+te_{3}+\left(t-\lambda t+\lambda-\frac{1}{t}\right)e_{5}+\frac{1}{4}\left(\frac{1-t^{2}}{t}\right)e_{6}+\left(t-\lambda t+\lambda-\frac{1}{t}\right)e_{7}&\\
\mu_{1}(g_{t}e_{1},g_{t}e_{2})&=\mu_{1}\left(te_{1}+\frac{1}{4}\left(\frac{t^{2}-1}{t}\right)e_{5},e_{2}+\left(1-\lambda+\frac{\lambda}{t}-\frac{1}{t^{2}}\right)e_{5}\right)&\\
                              &=t(e_{3}+e_{2})+t\left(1-\lambda+\frac{\lambda}{t}-\frac{1}{t^{2}}\right)(e_{7}+e_{5})-\frac{1}{4}\left(\frac{t^{2}-1}{t}\right)e_{6}&\\
                              &=te_{2}+te_{3}+\left(t-\lambda t+\lambda-\frac{1}{t}\right)e_{5}+\frac{1}{4}\left(\frac{1-t^{2}}{t}\right)e_{6}+\left(t-\lambda t+\lambda-\frac{1}{t}\right)e_{7}&
\end{flalign*}
\end{footnotesize}\\[-0.8cm]
\begin{footnotesize}
\begin{flalign*}
\bullet \ g_{t}\mu_{t}(e_{1},e_{3})&=g_{t}e_{4}&\\
                                   &=t^{2}e_{4}+\frac{1}{2}(1-t^{2})e_{6}+\left(\frac{1}{2}t^{2}-\lambda t^{2}+\lambda t-\frac{1}{2}\right)e_{7}&\\
\mu_{1}(g_{t}e_{1},g_{t}e_{3})&=\mu_{1}\left(te_{1}+\frac{1}{4}\left(\frac{t^{2}-1}{t}\right)e_{5},te_{3}+\frac{1}{4}\left(\frac{1-t^{2}}{t}\right)e_{6}+\left(t-\lambda t+\lambda-\frac{1}{t}\right)e_{7}\right)&\\
                              &=t^{2}e_{4}+\frac{1}{4}t\left(\frac{1-t^{2}}{t}\right)(e_{7}+2e_{6})+t\left(t-\lambda t+\lambda-\frac{1}{t}\right)e_{7}-\frac{1}{4}(t^{2}-1)e_{7}&\\
                              &=t^{2}e_{4}+\frac{1}{4}(1-t^{2})e_{6}+\left(\frac{1}{4}-\frac{1}{4}t^{2}+t^{2}-\lambda t^{2}+\lambda t-1-\frac{1}{4}t^{2}+\frac{1}{4} \right)e_{7}&\\
                              &=t^{2}e_{4}+\frac{1}{4}(1-t^{2})e_{6}+\left(\frac{1}{2}t^{2}-\lambda t^{2}+\lambda t-\frac{1}{2}\right)e_{7}&
\end{flalign*}
\end{footnotesize}\\[-0.8cm]
\begin{footnotesize}
\begin{flalign*}
\bullet \ g_{t}\mu_{t}(e_{1},e_{4})&=g_{t}(e_{6}+\lambda e_{7})&\\
                                   &=te_{6}+\lambda t^{2}e_{7}&\\
\mu_{1}(g_{t}e_{1},g_{t}e_{4})&=\mu_{1}\left(te_{1}+\frac{1}{4}\left(\frac{t^{2}-1}{t}\right)e_{5},t^{2}e_{4}+\frac{1}{2}(1-t^{2})e_{6}+\left(\frac{1}{2}t^{2}-\lambda t^{2}+\lambda t-\frac{1}{2}\right)e_{7}\right)&\\
                              &=t^{3}(e_{6}+\lambda e_{7})+\frac{1}{2}t(1-t^{2})(e_{7}+2e_{6})+t\left(\frac{1}{2}t^{2}-\lambda t^{2}+\lambda t-\frac{1}{2}\right)e_{7}&\\
                              &=(t^{3}+t-t^{3})e_{6}+\left(\lambda t^{3}+\frac{1}{2}t-\frac{1}{2}t^{3}+\frac{1}{2}t^{3}-\lambda t^{3}+\lambda t^{2}-\frac{1}{2}t\right)e_{7}&\\
                              &=te_{6}+\lambda t^{2}e_{7}&
\end{flalign*}
\end{footnotesize}\\[-0.8cm]
\begin{footnotesize}
\begin{flalign*}
\bullet \ g_{t}\mu_{t}(e_{1},e_{5})&=g_{t}(e_{7}+te_{5})&\\
                                   &=t^{2}e_{7}+t^{2}e_{5}+t\left(\lambda t-t-\lambda +\frac{1}{t}\right)e_{7}&\\  
                                   &=t^{2}e_{5}+(\lambda t^{2}-\lambda t+1)e_{7}&\\
\mu_{1}(g_{t}e_{1},g_{t}e_{5})&=\mu_{1}\left(te_{1}+\frac{1}{4}\left(\frac{t^{2}-1}{t}\right)e_{5},te_{5}+\left(\lambda t-t-\lambda+\frac{1}{t}\right)e_{7}\right)&\\
                              &=t^{2}(e_{7}+e_{5})+t\left(\lambda t-t-\lambda+\frac{1}{t}\right)e_{7}&\\
                              &=t^{2}e_{5}+(t^{2}+\lambda t^{2}-t^{2}-\lambda t+1)e_{7}&\\
                              &=t^{2}e_{5}+(\lambda t^{2}-\lambda t+1)e_{7}&          
\end{flalign*}
\end{footnotesize}\\[-0.8cm]
\begin{footnotesize}
\begin{flalign*}
\bullet \ g_{t}\mu_{t}(e_{1},e_{6})&=g_{t}(e_{7}+2te_{6})&\\
                                   &=t^{2}e_{7}+2t^{2}e_{6}&\\  
\mu_{1}(g_{t}e_{1},g_{t}e_{6})&=\mu_{1}\left(te_{1}+\frac{1}{4}\left(\frac{t^{2}-1}{t}\right)e_{5},te_{6}\right)&\\
                              &=t^{2}(e_{7}+2e_{6})&\\
                              &=t^{2}e_{7}+2t^{2}e_{6}&     
\end{flalign*}
\end{footnotesize}\\[-0.8cm]
\begin{footnotesize}
\begin{flalign*}
\bullet \ g_{t}\mu_{t}(e_{1},e_{7})&=g_{t}(te_{7})&\\
                                   &=t^{3}e_{7}&\\  
\mu_{1}(g_{t}e_{1},g_{t}e_{7})&=\mu_{1}\left(te_{1}+\frac{1}{4}\left(\frac{t^{2}-1}{t}\right)e_{5},t^{2}e_{7}\right)&\\
                              &=t^{3}e_{7}& 
\end{flalign*}
\end{footnotesize}\\[-0.8cm]
\begin{footnotesize}
\begin{flalign*}
\bullet \ g_{t}\mu_{t}(e_{2},e_{3})&=g_{t}e_{5}&\\
                                   &=te_{5}+\left(\lambda t-t-\lambda+\frac{1}{t}\right)e_{7}&\\  
\mu_{1}(g_{t}e_{2},g_{t}e_{3})&=\mu_{1}\left(e_{2}+\left(1-\lambda+\frac{\lambda}{t}-\frac{1}{t^{2}}\right)e_{5},te_{3}+\frac{1}{4}\left(\frac{1-t^{2}}{t}\right)e_{6}+\left(t-\lambda t+\lambda-\frac{1}{t}\right)e_{7}\right)&\\
                              &=te_{5}-t\left(1-\lambda +\frac{\lambda}{t}-\frac{1}{t^{2}}\right)e_{7}&\\
                              &=te_{5}+\left(\lambda t-t-\lambda +\frac{1}{t}\right)e_{7}&
\end{flalign*}
\end{footnotesize}\\[-0.8cm]
\begin{footnotesize}
\begin{flalign*}
\bullet \ g_{t}\mu_{t}(e_{2},e_{4})&=g_{t}e_{7}&\\
                                   &=t^{2}e_{7}&\\  
\mu_{1}(g_{t}e_{2},g_{t}e_{4})&=\mu_{1}\left(e_{2}+\left(1-\lambda+\frac{\lambda}{t}-\frac{1}{t^{2}}\right)e_{5},t^{2}e_{4}+\frac{1}{4}(1-t^{2})e_{6}+\left(\frac{1}{2}t^{2}-t+\lambda+\frac{1}{t}\right)e_{7}\right)&\\
                              &=t^{2}e_{7}&
\end{flalign*}
\end{footnotesize}\\[-0.8cm]
\begin{footnotesize}
\begin{flalign*}
\bullet \ g_{t}\mu_{t}(e_{2},e_{5})&=g_{t}e_{6}&\\
                                   &=te_{6}&\\  
\mu_{1}(g_{t}e_{2},g_{t}e_{5})&=\mu_{1}\left(e_{2}+\left(1-\lambda+\frac{\lambda}{t}-\frac{1}{t^{2}}\right)e_{5},te_{5}+\left(\lambda t-t-\lambda+\frac{1}{t}\right)e_{7}\right)&\\
                              &=te_{6}&
\end{flalign*}
\end{footnotesize}\\[-0.8cm]
\begin{footnotesize}
\begin{flalign*}
\bullet \ g_{t}\mu_{t}(e_{3},e_{5})&=g_{t}e_{7}&\\
                                   &=t^{2}e_{7}&\\  
\mu_{1}(g_{t}e_{3},g_{t}e_{5})&=\mu_{1}\left(te_{3}+\frac{1}{4}\left(\frac{1-t^{2}}{t}\right)e_{6}+\left(t-\lambda t+\lambda-\frac{1}{t}\right)e_{7},te_{5}+\left(\lambda t-t-\lambda+\frac{1}{t}\right)e_{7}\right)&\\
                              &=t^{2}e_{7}&
\end{flalign*}
\end{footnotesize}\\

For the remaining seven algebras $\mu$ we give, in the table below, the ideal $\h$ of codimension 1, the semisimple derivation $D\in Der(\h)$
that we choose to construct the linear deformation $\mu_t$, and the family $g_{t}\in GL_{7}$ satisfying
$g_t^{-1}\cdot\mu_1=\mu_t$, which is not difficult to check by hand.  
Therefore $\mu_1\rightarrow_{\deg}\mu$ and the proof is complete.
\end{proof}

\begin{center}
\begin{longtable}{|c|c|c|c|}\hline
\scriptsize{$\g$} & \scriptsize{$\h$} & \scriptsize{$D\in Der(\h)$} & \scriptsize{$g_{t}$}\\ \hline 
\rule[-1.5cm]{0cm}{3.2cm}\scriptsize{$\g_{7,0.1}$} & \scriptsize{$\langle e_{1},e_{3},e_{4},e_{5},e_{6},e_{7} \rangle $} & \scriptsize{$\left(\begin{smallmatrix}
1 &   &   &   &   &   \\
  & 3 &   &   &   &   \\
  &   & 4 &   &   &   \\
  &   &   & 5 &   &   \\
  &   &   &   & 6 &   \\
  &   &   &   &   & 7 \\
\end{smallmatrix}\right)$} & \scriptsize{$\left(\begin{smallmatrix} 
1 & 0 & 0 & 0 & 0 & 0 & 0\\[0.1cm]
0 & t & 0 & 0 & 0 & 0 & 0\\[0.1cm]
\frac{1}{2}\left(\frac{t-1}{t} \right) & 0 & 1 & 0 & 0 & 0 & 0\\[0.1cm]
0 & 0 & 0 & 1 & 0 & 0 & 0\\[0.1cm]
0 & \frac{1}{6}\left(\frac{3t^{2}-5t+2}{t} \right) & 0 & 0 & 1 & 0 & 0\\[0.1cm]
0 & 0 & \frac{1}{3}\left(\frac{1-t}{t} \right) & 0 & 0 & 1 & 0\\[0.1cm]
0 & 0 & 0 & \frac{1}{3}\left(\frac{1-t}{t} \right) & \frac{1}{2}\left(\frac{1-t}{t} \right) & 0 & 1\\
\end{smallmatrix}\right)$}\\ \hline
\rule[-1.6cm]{0cm}{3.4cm}\scriptsize{$\g_{7,0.2}$} & \scriptsize{$\langle e_{1},e_{3},e_{4},e_{5},e_{6},e_{7} \rangle $} & \scriptsize{$\left(\begin{smallmatrix}
1 &   &   &   &   &  \\
  & 0 &   &   &   &  \\
  &   & 1 &   &   &  \\
  &   &   & 2 &   &  \\
  &   &   &   & 3 &  \\
  &   &   &   &   & 4\\
\end{smallmatrix}\right)$} & \scriptsize{$\left(\begin{smallmatrix} 
1 & 0 & 0 & 0 & 0 & 0 & 0\\[0.1cm]
0 & t & 0 & 0 & 0 & 0 & 0\\[0.1cm]
0 & 0 & t & 0 & 0 & 0 & 0\\[0.1cm]
0 & \frac{1}{8}\left(\frac{4t-3t^{2}-1}{t} \right) & 0 & t & 0 & 0 & 0\\[0.1cm]
\frac{1}{8}\left(\frac{t^{2}-1}{t^{2}}\right) & 0 & \frac{1}{2}(1-t) & 0 & t & 0 & 0\\[0.1cm]
0 & 0 & 0 & \frac{1}{2}(1-t) & 0 & t & 0\\[0.1cm]
0 & 0 & \frac{1}{8}\left(\frac{1-t^{2}}{t}\right) & 0 & \frac{1}{2}(1-t) & 0 & t\\
\end{smallmatrix}\right)$}\\ \hline
\rule[-1.3cm]{0cm}{2.8cm}\scriptsize{$\g_{7,0.3}$} & \scriptsize{$\langle e_{1},e_{3},e_{4},e_{5},e_{6},e_{7} \rangle$} & \scriptsize{$\left(\begin{smallmatrix}
1 &   &   &   &   &  \\
  & 0 &   &   &   &  \\
  &   & 1 &   &   &  \\
  &   &   & 2 &   &  \\
  &   &   &   & 3 &  \\
  &   &   &   &   & 4\\
\end{smallmatrix}\right)$} & \scriptsize{$\left(\begin{smallmatrix}
1 & 0 & 0 & 0 & 0 & 0 & 0\\[0.1cm]
0 & t & 0 & 0 & 0 & 0 & 0\\[0.1cm]
0 & 0 & t & 0 & 0 & 0 & 0\\[0.1cm]
\frac{1}{4}\left(\frac{t-1}{t} \right) & 0 & 0 & t & 0 & 0 & 0\\[0.1cm]
0 & \frac{1}{3}(1-t) & 0 & 0 & t & 0 & 0\\[0.1cm]
0 & 0 & \frac{1}{3}(1-t) & 0 & 0 & t & 0\\[0.1cm]
0 & 0 & \frac{1}{4}(1-t) & \frac{1}{3}(1-t) & 0 & 0 & t\\
\end{smallmatrix}\right)$}\\ \hline
\rule[-1.7cm]{0cm}{3.6cm}\scriptsize{$\g_{7,0.5}$} & \scriptsize{$\langle e_{1},e_{3},e_{4},e_{5},e_{6},e_{7} \rangle$} & \scriptsize{$\left(\begin{smallmatrix}
1 &   &   &   &   &  \\
  & 0 &   &   &   &  \\
  &   & 1 &   &   &  \\
  &   &   & 3 &   &  \\
  &   &   &   & 2 &  \\
  &   &   &   &-1 & 3\\
\end{smallmatrix}\right)$} & \scriptsize{$\left(\begin{smallmatrix}
1 & 0 & 0 & 0 & 0 & 0 & 0\\[0.1cm]
0 & t & 0 & 0 & 0 & 0 & 0\\[0.1cm]
0 & 0 & t & 0 & 0 & 0 & 0\\[0.1cm]
0 & \frac{1}{3}\left(\frac{t^{2}-1}{t} \right) & 0 & t & 0 & 0 & 0\\[0.1cm]
\frac{1}{6}\left(\frac{t^{2}-1}{t^{2}} \right) & 0 & \frac{1}{3}\left(\frac{1-t^{2}}{t} \right) & 0 & 1 & 0 & 0\\[0.1cm]
0 & 0 & \frac{1}{6}\left(\frac{t^{2}-1}{t} \right) & 0 & 0 & t & 0\\[0.1cm]
0 & 0 & \frac{1}{3}\left(\frac{t^{2}-1}{t} \right) & 0 & \frac{5}{6}(t^{2}-1) & 0 & t\\
\end{smallmatrix}\right)$}\\ \hline
\rule[-1.5cm]{0cm}{3.2cm}\scriptsize{$\g_{7,0.6}$} & \scriptsize{$\langle e_{2},e_{3},e_{4},e_{5},e_{6},e_{7} \rangle$} & \scriptsize{$\left(\begin{smallmatrix}
1 &   &   &   &   &  \\
  & 0 &   &   &   &  \\
  &   & 2 &   &   &  \\
  &   &   & 1 &   &  \\
  &   &   &   & 3 &  \\
  &   &   &   &   & 2\\
\end{smallmatrix}\right)$} & \scriptsize{$\left(\begin{smallmatrix}
t & 0 & 0 & 0 & 0 & 0 & 0\\[0.1cm]
0 & 1 & 0 & 0 & 0 & 0 & 0\\[0.1cm]
0 & 0 & t & 0 & 0 & 0 & 0\\[0.1cm]
0 & \frac{1}{2}\left(\frac{1-t^{2}}{t^{2}} \right) & \frac{1}{2}\left(\frac{1-t^{2}}{t} \right) & 1 & 0 & 0 & 0\\[0.1cm]
0 & 0 & 0 & 0 & t & 0 & 0\\[0.1cm]
0 & 0 & 0 & 0 & \frac{1}{2}\left(\frac{1-t^{2}}{t} \right) & 1 & 0\\[0.1cm]
0 & 0 & \frac{1}{2}\left(\frac{1-t^{2}}{t} \right) & \frac{3}{2}(1-t^{2}) & \frac{1}{2}\left(\frac{t^{2}-1}{t} \right) & 0 & t\\
\end{smallmatrix}\right)$}\\ \hline
\rule[-1.1cm]{0cm}{2.4cm}\scriptsize{$\g_{7,0.7}$} & \scriptsize{$\langle e_{2},e_{3},e_{4},e_{5},e_{6},e_{7} \rangle$} & \scriptsize{$\left(\begin{smallmatrix}
1 &   &   &   &   &  \\
  & 0 &   &   &   &  \\
  &   & 0 &   &   &  \\
  &   &   & 1 &   &  \\
  &   &   &   & 2 &  \\
  &   &   &   &   & 1\\
\end{smallmatrix}\right)$} & \scriptsize{$\left(\begin{smallmatrix}
t & 0 & 0 & 0 & 0 & 0 & 0\\[0.1cm]
0 & 1 & 0 & 0 & 0 & 0 & 0\\[0.1cm]
0 & 0 & t & 0 & 0 & 0 & 0\\[0.1cm]
(t-1) & 0 & 0 & t^{2} & 0 & 0 & 0\\[0.1cm]
0 & 0 & 0 & 0 & t & 0 & 0\\[0.1cm]
0 & 0 & 0 & 0 & 0 & t & 0\\[0.1cm]
0 & 0 & (1-t) & (1-t)t & 0 & 0 & t^{2}\\
\end{smallmatrix}\right)$}\\ \hline 
\rule[-1.2cm]{0cm}{2.6cm}\scriptsize{$\g_{7,0.8}$} & \scriptsize{$\langle e_{2},e_{3},e_{4},e_{5},e_{6},e_{7} \rangle$} & \scriptsize{$\left(\begin{smallmatrix}
1 &   &   &   &   &  \\
  & 0 &   &   &   &  \\
  &   & 0 &   &   &  \\
  &   &   & 2 &   &  \\
  &   &   &   & 1 &  \\
  &   &   &   &   & 2\\
\end{smallmatrix}\right)$} & \scriptsize{$\left(\begin{smallmatrix}
t & 0 & 0 & 0 & 0 & 0 & 0\\[0.1cm]
0 & 1 & 0 & 0 & 0 & 0 & 0\\[0.1cm]
0 & (1-t^{2}) & t^{3} & t(t^{2}-1) & 0 & 0 & 0\\[0.1cm]
0 & 0 & 0 & t & 0 & 0 & 0\\[0.1cm]
0 & 0 & 0 & 0 & t^{2} & 0 & 0\\[0.1cm]
0 & 0 & 0 & 0 & 0 & t^{3} & 0\\[0.1cm]
0 & 0 & \frac{1}{2}t^{2}(1-t) & t(1-t^{2}) & t^{2}(1-t^{2}) & 0 & t^{3}\\
\end{smallmatrix}\right)$}\\ \hline
\end{longtable}
\end{center}

\noindent{\bf Acknowledgements.}
This paper is part of the PhD.\ thesis of the first author. 
He thanks CONICET for the Ph.D.\ fellowship awarded that made this possible.
We thank Oscar Brega, Leandro Cagliero and Edison Fern\'andez-Culma for
their comments that helped us improved the presentation of this paper.



\begin{thebibliography}{AGGV}

\bibitem[AG]{AG} Ancochea-Bermudez, J.\ and Goze, M.,
  \emph{On the varieties of nilpotent Lie algebras of dimension 7 and 8},
    J.\ Pure Appl.\ Algebra 77 (1992), 131--140.

\bibitem[AGGV]{AGGV} Ancochea-Bermudez, J.,\ G\'{o}mez-Martin, J.,\ Goze, M., and Valeiras G.,
  \emph{Sur les composantes irr\'{e}ductibles de la variet\'{e} des lois d'alg\`{e}bres de Lie nilpotentes},
    J.\ Pure Appl.\ Algebra 106 (1996), 11--22.  
    
\bibitem[B]{B} Burde, D.,
   \emph{Degenerations of 7-dimensional nilpotent Lie algebras}, 
   Comm.\ Algebra 33 (2005), no.\ 4, 1259--1277.  
   
\bibitem[BSt]{BSt} Burde, D.\ and Steinhoff C.,
   \emph{Classification of orbit closures of 4-dimensional complex Lie algebras}, 
   J.\ Algebra 214 (1999), no.\ 2, 729--739.     

\bibitem[C1]{C1} Carles, R., 
   \emph{Sur la structure des alg\`{e}bres de Lie rigides},
   Ann.\ Inst.\ Fourier 34 (1984), no.\ 3, 65--82.

\bibitem[C2]{C2} Carles, R., 
   \emph{Weight systems for complex nilpotent Lie algebras and applications to the varieties of Lie algebras}
   Publ.\ Univ.\ Poitiers, (1996).

\bibitem[CD]{CD} Carles, R., Diakit\'{e}
   \emph{Sur les vari\'{e}t\'{e}s d'alg\`{e}bres de Lie de dimension $\leq 7$},
   J.\ Algebra 91 (1984), no.\ 1, 53--63.

\bibitem[G]{G} Gerstenhaber, M.,
   \emph{On the deformations of rings and algebras}, 
   Ann.\ Math. 74 (1964), no. 1, 59--103.

\bibitem[GS]{GS}  Gerstenhaber, M.\ and Schack, S.,
  \emph{Relative Hochschild cohomology, rigid algebras, and the Bockstein},
   J.\ Pure Appl.\ Algebra 43 (1986), no.\ 1, 53--74.    

\bibitem[GO1]{GO1} Grunewald, F.\ and O'Halloran, J.,
   \emph{Varieties of nilpotent Lie algebras of dimension less than six}, 
   J.\ Algebra 112 (1988), no.\ 2, 315--325.

\bibitem[GO2]{GO2} Grunewald, F.\ and O'Halloran, J.,
   \emph{Deformations of Lie algebras}, 
   J.\ Algebra 162 (1993), no.\ 1, 210--224. 
   
\bibitem[J]{J} Jacobson, N.,
   \emph{Lie algebras}, 
   Dover Publications, Inc., New York (1979).    
   
\bibitem[K]{K}  Khakimdjanov, Y.,
    \emph{Characteristically nilpotent Lie algebras}
    Math.\ USSR Sbornik 70 (1990),no. 1, 65--78.

\bibitem[M]{M} Magnin, L.,
 \emph{Determination of 7-dimensional indecomposable nilpotent complex Lie algebras by adjoining a derivation to 6-dimensional Lie algebras},
 Algebr.\ Represent.\ Theory 13 (2010), no.\ 6, 723--753.
 
\bibitem[NR1]{NR1} Nijenhuis, A. and Richardson R.W.,
   \emph{Cohomology and deformations in graded Lie algebras}, 
   Bull.\ Amer.\ Math.\ Soc. 72 (1966), 1--29.

\bibitem[NR2]{NR2} Nijenhuis, A. and Richardson R.W.,
   \emph{Deformations of Lie algebra structures}, 
   J.\ Math.\ Mech. 17 (1967), 89--105.

\bibitem[R]{R} Richardson, R.W.,
   \emph{On the rigidity of semi-direct products of Lie algebras}, 
   Pacific.\ J.\ Math. 22 (1967), no. 2, 339--344.

\bibitem[Se]{Se} Seeley, C.,
   \emph{Degenerations of 6-dimensional nilpotent Lie algebras over $\mathbb{C}$}, 
   Comm.\ Algebra 18 (1990), no.\ 10, 3493--3505.  

\bibitem[St]{St} Steinhoff C.,
   \emph{Klassifikation und Degeneration von Lie Algebren}, 
   Diplomarbeit, D\"{u}sseldorf (1997).

\bibitem[V]{V} Vergne, M., 
   \emph{Cohomologie des alg\`{e}bres de Lie nilpotentes}, 
   Bull. Soc. Math. France, vol. 98, 81--116. 

\end{thebibliography}
\end{document}